\documentclass[reqno]{amsproc}
\usepackage[utf8]{inputenc}
\usepackage{cite}
\usepackage{xcolor}
\usepackage{calligra}
\usepackage[T1]{fontenc}
\usepackage{latexsym}
\usepackage[all]{xy}
\usepackage{mathrsfs}
\usepackage{algorithmic}
\usepackage{amsmath,amsfonts,amssymb,amsxtra}
\numberwithin{equation}{section}
\newtheorem{theorem}{\bf Theorem}[section]

\newtheorem{lem}{\bf Lemma}[section]
\newtheorem{cor}{\bf Corollary}[section]
\newtheorem{remark}{\bf Remark}[section]

\newcommand\dv{\mathrm{div}}
\newcommand\tr{\mathrm{tr}}

\usepackage{hyperref}

\begin{document}

\title[Estimates of eigenvalues of elliptical differential problems]{Estimates of eigenvalues of elliptical differential problems in divergence form}

\author{Marcio C. Ara\'ujo Filho$^1$}
\author{Juliana F.R. Miranda$^2$}
\author{Cristiano S. Silva$^3$}

\address{$^1$Departamento de Matemática, Universidade Federal de Rondônia, Campus Ji-Paraná, R. Rio Amazonas, 351, Jardim dos Migrantes, 76900-726 Ji-Paraná, Rondônia, Brazil}
\address{$^{2}$Departamento de Matem\'atica, Universidade Federal do Amazonas, Av. General Rodrigo Oct\'avio, 6200, 69080-900 Manaus, Amazonas, Brazil.}
\address{$^{3}$Departamento de Matem\'atica, Universidade Federal de Roraima, Av. Ene Garcez, 2413, 69310-000 Boa Vista, Roraima, Brazil.}

\email{$^1$marcio.araujo@unir.br}
\email{$^2$jfrmiranda@ufam.edu.br}
\email{$^3$cristiano.silva@ufrr.br}

\urladdr{$^1$https://dmejp.unir.br}
\urladdr{$^{2}$http://dmice.ufam.edu.br}
\urladdr{$^{3}$http://ufrr.br/matematicalicenciatura/}



\begin{abstract}

In this paper, we compute universal estimates of eigenvalues for a class of coupled systems of elliptic differential equations in divergence form on a bounded domain in Euclidean space, which includes the well-known Lamé and the Laplacian operator. Furthermore, we also give universal estimates of eigenvalues for a class of fourth-order elliptic differential problems in divergence form, which encloses the well-known bi-Laplacian operator. In both cases, as applications, we obtain the gap between consecutive eigenvalues as well as an upper bound for each eigenvalue. 

\end{abstract}

\maketitle

\section{Introduction}

Let $(M^n,\langle ,\rangle )$ be an $n$-dimensional complete Riemannian manifold and $\Omega\subset M$ be a bounded and connected domain with smooth boundary $\partial \Omega$. Let us consider a symmetric positive definite $(1, 1)$-tensor $T$ on $M$ and a real-valued function $\eta\in C^2(M)$, so that we define a second-order elliptic differential
operator $\mathscr{L}$ in the $(\eta, T)$-divergence form as follows:
\begin{equation}\label{EQ2.1}
    \mathscr{L}f :=\dv_\eta (T(\nabla f)) = \dv(T(\nabla f)) - \langle \nabla \eta, T(\nabla f) \rangle,
\end{equation}
where div stands for the divergence operator and $\nabla$ is the gradient operator. In these conditions, there are positive real numbers $\varepsilon$ and $\delta$ such that $\varepsilon I \leq T \leq \delta I$ where $I$ is the identity $(1, 1)$-tensor.

Elliptic differential operators in divergence form play a crucial role in various physical theories and applications. The eigenvalues of differential operators can provide significant insights into the physical and geometric properties of a domain. One of the reasons they are so interesting is that they involve many fields of mathematics such as spectral theory, Riemannian geometry, and partial differential equations. Over the years, many mathematicians have investigated properties of the spectrum of differential operators, some of which will be mentioned at the appropriate time.

It is worth noting that geometric motivations for studying the operator $\mathscr{L}$ in the $(\eta,T)$-divergence form on bounded domains of Riemannian manifolds were provided in the work of Gomes and Miranda~\cite[Section~2]{GomesMiranda}. They demonstrated that this operator arises as the trace of a $(1,1)$-tensor on a Riemannian manifold $M$, and established a Bochner-type formula for it. A remarkable point is that Eq.~ $(2.3)$ in \cite{GomesMiranda} connects $\mathscr{L}$ with the operator $\square$ introduced by Cheng and Yau in~\cite{ChengYau}. In particular, when $T$ is chosen as the Einstein tensor, this relation can open the door to potential applications in physics.

In Section~\ref{sec_1} we compute eigenvalue estimates for a coupled system of second-order elliptic differential equations in divergence form and  in Section~\ref{sec_2} a fourth-order elliptic differential problem in divergence form.

\subsection{A coupled system of second-order elliptic differential equations}\label{sec_1}

In this section, let us consider $(\mathbb{R}^n,\cdot )$ the $n$-dimensional Euclidean space with its canonical metric, we address the eigenvalue problem for an operator which is a second-order perturbation of $\mathscr{L}$. More precisely, we compute universal estimates of the eigenvalues of the coupled system of second-order elliptic differential equations, namely:
\begin{equation}\label{problem1}
    \left\{\begin{array}{ccccc} 
    \mathscr{L}  {\bf u} + \alpha \nabla(\dv_\eta {\bf u}) &=& -\sigma {\bf u} & \mbox{in } & \Omega,\\
     {\bf u}&=&0 & \mbox{on} & \partial\Omega, 
    \end{array} 
    \right.
\end{equation}
where ${\bf u}=(u^1, u^2, \ldots, u^n)$ is a vector-valued function from $\Omega$ to $\mathbb{R}^n$, the constant $\alpha$ is non-negative and $\mathscr{L}{\bf u}=(\mathscr{L}u^1, \mathscr{L}u^2, \ldots,\mathscr{L} u^n)$. 

According to presented in Araújo Filho and Gomes~\cite{AF-G}, we can see that $\mathscr{L}+\alpha\nabla\dv_\eta$ is a formally self-adjoint operator in the Hilbert space $\mathbb{L}^2(\Omega,dm)$ of all vector-valued functions that vanish on $\partial \Omega$ in the sense of the trace, where $dm=e^{-\eta}d\Omega$. Moreover, the eigenvalue problem~\eqref{problem1} has a real and discrete spectrum 
\begin{equation}\label{equationn1.3}
    0 < \sigma_1 \leq \sigma_2 \leq \cdots \leq \sigma_k \leq \cdots\to\infty,
\end{equation}
where each $\sigma_i$ is repeated according to its multiplicity.

Some cases can be obtained from Problem~\ref{problem1} when $T$ is a special type of tensor. For example, when $T$ is divergence-free, we obtain Problem~\ref{problem1-3} and when $T$ is the identity tensor, we get \ref{problem4}. In addition, if $\eta$ is a constant function, we obtain the problem for Cheng-Yau and Laplacian operators, respectively. In the next three paragraphs, we make brief discussions about them.

When $T$ is divergence-free the operator $\mathscr{L}$ can be decomposed as follows
\begin{equation}\label{DCYOp}
    \mathscr{L}f = \square f - \nabla \eta\cdot T(\nabla f),
\end{equation}
where $\square f = \tr{(\nabla^2f \circ T)}=\langle \nabla^2 f, T\rangle$ is the operator introduced by Cheng and Yau~\cite{ChengYau}, known as the Cheng-Yau operator, that arises from the study of complete hypersurfaces of constant scalar curvature in space forms (divergence-free tensors also often arise from physical facts, see e.g. Serre~\cite{Serre}). Eq.~\eqref{DCYOp} is a first-order perturbation of the Cheng-Yau operator, and it defines the drifted Cheng-Yau operator which we denote by $\square_\eta$ with a drifting function $\eta$. In this case, the coupled system of second-order elliptic differential equations \eqref{problem1} is rewritten as
\begin{equation}\label{problem1-3}
    \left\{\begin{array}{ccccc} 
    \square_\eta  {\bf u} + \alpha \nabla(\dv_\eta {\bf u}) &=& -\sigma {\bf u} & \mbox{in } & \Omega,\\
     {\bf u}&=&0 & \mbox{on} & \partial\Omega, 
    \end{array} 
    \right.
\end{equation}
where ${\bf u}=(u^1, u^2, \ldots, u^n)$ is a vector-valued function from $\Omega$ to $\mathbb{R}^n$, the constant $\alpha$ is non-negative and $\square_\eta{\bf u}:=(\square_\eta u^1, \square_\eta u^2, \ldots,\square_\eta  u^n)$. Problem~\eqref{problem1-3} was studied by first author and Gomes~\cite{AF-G} which obtained some bounds of eigenvalues for it.

Also, if $T=I$ and $\eta$ is not necessarily constant, then Problem~\eqref{problem1} becomes
\begin{equation}\label{problem4}
    \left\{ \begin{array}{ccccc} 
    \Delta_\eta  {\bf u} + \alpha \nabla \mbox{(div}_\eta{\bf u}) &=& -\sigma {\bf u} & \mbox{in} &\Omega,  \\
     {\bf u}&=&0 & \mbox{on} & \partial \Omega,
    \end{array} 
    \right.
\end{equation}
where $\Delta_{\eta}{\bf u}=(\Delta_\eta u^1,\ldots,\Delta_\eta u^n)$ and $\Delta_\eta = \dv_\eta \nabla$ is the drifted Laplacian operator on $C^\infty(\Omega)$. Du and Bezerra~\cite{DuBezerra} established some bounds of the eigenvalues for it, and the first author and Gomes~\cite{AF-G} obtained better bounds than their results. 

Additionally when $\eta$ is a constant function in Problem \eqref{problem4} the operator becomes $\Delta+\alpha\nabla \dv$ known as Lamé's operator. In the $3$-dimensional case it shows up in the elasticity theory  and $\alpha$ is determined by the positive constants of Lamé, so the assumption $\alpha \geq 0$ is justified, the interested reader can consult Pleijel~\cite{Pleijel} or Kawohl and Sweers~\cite{KawohlSweers}. It is important to emphasize here the works of Levine and Protter~\cite{LevineProtter}, Livitin and Parnovski~\cite{LevitinParnovski}, Hook~\cite{Hook}, Cheng and Yang~\cite{ChengYang} and Chen et al.~\cite{CCWX} in which we can find some estimates of the eigenvalues of this case.

Our first result is a universal quadratic estimate for the eigenvalues of Problem~\eqref{problem1}, which will play a key role to obtain some of our estimates.
\begin{theorem}\label{theorem1.1}
Let $\Omega \subset \mathbb{R}^n$ be a  bounded domain and ${\bf u}_i$ be a normalized eigenfunction corresponding to $i$-th eigenvalue $\sigma_i$ of Problem~\eqref{problem1}. Then, for any positive integer $k$, we get
\begin{equation*}
    \sum_{i=1}^k(\sigma_{k+1}-\sigma_i)^2 \leq \frac{4\delta(n\delta+\alpha)}{n^2\varepsilon^2}\sum_{i=1}^k(\sigma_{k+1}-\sigma_i)\Big(\sigma_i - \alpha \|\dv_\eta{\bf u}_i\|^2 + \frac{T_0^2+4C_0}{4\delta}\Big),
\end{equation*}
where
\begin{equation*}
    C_0=\sup_\Omega \left\{\frac{1}{2}\dv \Big( T\big(T(\nabla \eta)-\tr(\nabla T)\big)\Big) - \frac{1}{4}|T(\nabla \eta)|^2\right\},
\end{equation*}
$T_0=\sup_{\Omega}|\tr(\nabla T)|$ and $\eta_0=\sup_{\Omega}|\nabla \eta|$.
\end{theorem} 

The Theorem~\ref{theorem1.1} can be seen as an extension of the Yang's estimate of the eigenvalues of the Laplacian on real-valued functions. Its proof is a convenient algebraic manipulation of \cite[Theorem 1.1]{AF-G} obtained by the first author and Gomes. 

Considering the sum of lower order eigenvalues in terms of the first eigenvalue and its corresponding eigenfunction, we prove the following estimate.

\begin{theorem}\label{theorem1.2}
Let $\Omega \subset \mathbb{R}^n$ be a bounded domain, $\sigma_i$ be the $i$-th eigenvalue of Problem~\eqref{problem1}, for $i=1,\ldots,n$, and ${\bf u}_1$ be a normalized eigenfunction corresponding to the first eigenvalue. Then, we get
\begin{equation*}
    \sum_{i=1}^n(\sigma_{i+1} - \sigma_1) \leq \frac{4\delta(\delta+\alpha)}{\varepsilon^2}(\sigma_1+D_1),
\end{equation*}
where $D_1= -\alpha\|\dv_\eta {\bf u}_1\|^2 + \frac{T_0+4C_0}{4\delta}$.
\end{theorem}

The Theorems~\ref{theorem1.1} and~\ref{theorem1.2} were obtained analogously to Theorems 1.1 and 1.2 in Araújo Filho and Gomes~\cite{AF-G}, however, with a more suitable configuration to obtain the eigenvalue estimates of Problem~\ref{problem1} for the more general case, fact that the mentioned authors achieved only for the case in which $T$ is divergence-free, see \cite[Corollaries~6.1 and 6.2]{AF-G}.

From Theorem~\ref{theorem1.1} and following the steps of the proof of \cite[Theorem 3]{GomesMiranda}, we obtain the next result.
\begin{cor}\label{Cor_1}
Under the same setup as in Theorem~\ref{theorem1.1}, and by defining $D_0=-\alpha \min_{j=1, \ldots, k} \|\dv_\eta{\bf u}_j\|^2 + \frac{T_0+4C_0}{4\delta} $, we have
\begin{equation*}
    \sum_{i=1}^k(\sigma_{k+1}-\sigma_i)^2 \leq \frac{4\delta(n\delta+\alpha)}{n^2\varepsilon^2}\sum_{i=1}^k(\sigma_{k+1}-\sigma_i)(\sigma_i +D_0),
\end{equation*}
and consequently, we get 
\begin{align*}
    \sigma_{k+1} + D_0 \leq& \Big(1+\frac{2\delta(n\delta+\alpha)}{\varepsilon^2n^2}\Big)\frac{1}{k}\sum_{i=1}^k(\sigma_i + D_0) + \Big[\Big( \frac{2\delta(n\delta+\alpha)}{\varepsilon^2n^2}\frac{1}{k}\sum_{i=1}^k(\sigma_i + D_0)\Big)^2 \nonumber \\
   &-\Big(1+\frac{4\delta(n\delta+\alpha)}{\varepsilon^2n^2}\Big) \frac{1}{k}\sum_{j=1}^k\Big(\sigma_j-\frac{1}{k}\sum_{i=1}^k\sigma_i\Big)^2 \Big]^{\frac{1}{2}}
\end{align*}
and
\begin{eqnarray*}
\nonumber\sigma_{k+1} -\sigma_k &\leq& 2\Big[\Big( \frac{2\delta(n\delta+\alpha)}{\varepsilon^2n^2}\frac{1}{k}\sum_{i=1}^k(\sigma_i + D_0)\Big)^2\\
&&-\Big(1+\frac{4\delta(n\delta+\alpha)}{\varepsilon^2n^2}\Big) \frac{1}{k}\sum_{j=1}^k\Big(\sigma_j-\frac{1}{k}\sum_{i=1}^k\sigma_i\Big)^2 \Big]^{\frac{1}{2}}.
\end{eqnarray*}
\end{cor}

From Corollary~\ref{Cor_1} and by applying the recursion formula of Cheng and Yang~\cite{ChengYang09}, we obtain the next corollary. 

\begin{cor}\label{corollary3.3-CY}
Under the same setup as in Corollary~\ref{Cor_1}, we have
\begin{equation*}
    \sigma_{k+1} + D_0 \leq \Big(1+\frac{4\delta(\delta n+\alpha)}{\varepsilon^2n^2}\Big)k^{\frac{2\delta( n\delta+\alpha)}{\varepsilon^2 n^2}}(\sigma_1 + D_0).
\end{equation*}
\end{cor}
\begin{remark}
    We observe that Corollaries \ref{Cor_1} and \ref{corollary3.3-CY} are extensions of the Corollaries~6.1 and~6.2 by the first author and Gomes \cite{AF-G}.  Since $T$ is divergence-free, we have $T_0=0$, and then we obtain their mentioned results.
    \end{remark}

\subsection{Eigenvalues of fourth-order elliptic operators in divergence form}\label{sec_2}

In this section, we get some inequalities for eigenvalues of a larger class of fourth-order elliptic operators in divergence form on Riemannian manifolds, which extend many operators of the literature, as the Dirichlet biharmonic operator. Namely, we will consider the following eigenvalue problem:
\begin{equation}\label{problem_1}
    \left\{\begin{array}{ccccc} 
    \mathscr{L}^2  u  &=&  \Gamma u & \mbox{in } & \Omega,\\
     u=\frac{\partial u}{\partial \nu_T}&=&0 & \mbox{on} & \partial\Omega, 
    \end{array} 
    \right.
\end{equation}
where $\frac{\partial u}{\partial \nu_T}=\langle T(\nabla u), \nu \rangle$ with $\nu$ the outward unit normal vector field of $\partial \Omega$. 

It is not difficult to verify that $\mathscr{L}^2$ is a formally self-adjoint operator in the space of all smooth real-valued functions $u$ such that $u|_{\partial \Omega}=\frac{\partial u}{\partial \nu_T}|_{\partial \Omega}=0$, with respect to the inner product
\begin{align*}
    \langle\langle u , v \rangle \rangle = \int_\Omega uv dm,
\end{align*}
where $dm = e^{-\eta}d\Omega$ is the weighted volume form on $\Omega$ (see \cite[Section~2]{AF}). Therefore the spectrum of Problem~\eqref{problem_1} is real and discrete, that is,
\begin{equation*}
    0 < \Gamma_1 \leq \Gamma_2 \leq \cdots \leq \Gamma_k \leq \cdots\to\infty,
\end{equation*}
where each $\Gamma_i$ is repeated according to its multiplicity.

Araújo Filho~\cite{AF}  give a generalized universal estimate to Problem~\ref{problem_1}, namely 
\begin{align*}
     \sum_{i=1}^k&(\Gamma_{k+1}-\Gamma_i)^2\nonumber\\
    \leq& \frac{4}{n\varepsilon}\Bigg\{\sum_{i=1}^k(\Gamma_{k+1}-\Gamma_i)^2\Big[\Big(\sqrt{\delta \Gamma_i^{\frac{1}{2}}}  + \frac{T_0}{2}\Big)^2 +\frac{n\delta \Gamma_i^{\frac{1}{2}}}{2} + \frac{n^2H_0^2+4\delta^2\overline{C}_0+2T_0T_*\eta_0}{4} \Big]\Bigg\}^{\frac{1}{2}}\nonumber\\
    &\times\Bigg\{\sum_{i=1}^k(\Gamma_{k+1}-\Gamma_i)\Big[\Big(\sqrt{\delta \Gamma_i^{\frac{1}{2}}}  + \frac{T_0}{2}\Big)^2 + \frac{n^2H_0^2+4\delta^2\overline{C}_0+2T_0T_*\eta_0}{4} \Big]\Bigg\}^{\frac{1}{2}},
\end{align*}
where
\begin{equation*}
    \overline{C}_0=\sup_{\Omega}{\Big \{}\frac{1}{2}\Delta \eta - \frac{1}{4}|\nabla \eta|^2{\Big \}},
\end{equation*}
$T_0=\sup_{\Omega}|\tr(\nabla T)|$, $\eta_0=\sup_{\Omega}|\nabla \eta|$, $T_*=\sup_{\Omega}|T|$, $H_0=\sup_\Omega |{\bf H}_T|$ and ${\bf H}_T$ is the generalized mean curvature vector of the immersion, see Eq.~\eqref{mean-curvature} for more details. The definition of the generalized mean curvature vector was considered by Grosjean~\cite{Grosjean} and Roth~\cite{Roth2}. In this general setup it is impossible to obtain the gap between consecutive eigenvalues. But, when $T$ is divergence-free, that is, in the bi-drifted Cheng–Yau operator case, the author got it, see \cite[Corollaries~1.1, 1.2 and 1.3]{AF}.

When $T$ is the identity operator, from ~\eqref{problem_1}, we have the problem
\begin{equation}\label{problem_3}
    \left\{\begin{array}{ccccc} 
    \Delta_\eta^2  u  &=& \Gamma u & \mbox{in } & \Omega,\\
     u=\frac{\partial u}{\partial \nu}&=&0 & \mbox{on} & \partial\Omega, 
    \end{array} 
    \right.
\end{equation}
where $\Delta_\eta^2$ is the {\it bi-drifted Laplacian operator} on $C^\infty(\Omega)$.  In this case, if $M$ is an $n$-dimensional submanifold isometrically immersed in a Euclidean space with mean curvature vector ${\bf H}$, Du et al.~\cite{Du-etal} proved the following eigenvalues inequality 
\begin{align*}
    \sum_{i=1}^k(\Gamma_{k+1}-\Gamma_i)^2 \leq& \frac{1}{n}\Bigg\{\sum_{i=1}^k(\Gamma_{k+1}-\Gamma_i)^2\Big[(2n+4)\Gamma_i^{\frac{1}{2}} +4\eta_0\Gamma_i^{\frac{1}{4}} + n^2H_0^2+\eta_0^2\Big]\Bigg\}^{\frac{1}{2}}\nonumber\\
    &\times\Bigg\{\sum_{i=1}^k(\Gamma_{k+1}-\Gamma_i)\Big(4\Gamma_i^{\frac{1}{2}}+4\eta_0\Gamma_i^{\frac{1}{4}} + n^2H_0^2+\eta_0^2 \Big)\Bigg\}^{\frac{1}{2}},
\end{align*}
where $H_0=\sup_\Omega |{\bf H}|$ and $\eta_0=\max_{\Bar{\Omega}}|\nabla \eta|$. We observe that their result is not suitable to obtain the gap between consecutive eigenvalues.

When $\eta$ is a constant function and $T$  is the identity operator the Problem~\ref{problem1} becomes
\begin{equation}\label{problem2}
    \left\{\begin{array}{ccccc} 
    \Delta^2  u  &=& \Gamma u & \mbox{in } & \Omega,\\
     u=\frac{\partial u}{\partial \nu}&=&0 & \mbox{on} & \partial\Omega, 
    \end{array} 
    \right.
\end{equation}
where $\Delta^2$ is the biharmonic operator on $C^\infty(\Omega)$. The Problem~\eqref{problem2} describes the characteristic vibrations of a clamped plate in elastic mechanics and is known as the {\it clamped plate problem} in the literature. 
To illustrate some inequalities of eigenvalues of the clamped plate problem, the interested reader is referred to \cite{Cheng-etal1}-\cite{ChengYang5}, \cite{HileYeh}, \cite{Hook}, \cite{Payne-etal}, and \cite{WangXia}. We present two of these inequalities below.

In the case where $M$ is an $n$-dimensional submanifold isometrically immersed in a Euclidean space, Problem~\eqref{problem2} was studied by Cheng et al.~\cite{Cheng-etal}, and if {\bf H} is the mean curvature vector of this immersion, they proved that
\begin{equation}\label{Cheng-etal-estimate}
    \sum_{i=1}^k(\Gamma_{k+1}-\Gamma_i)^2\leq \frac{1}{n^2}\sum_{i=1}^k(\Gamma_{k+1}-\Gamma_i)\Big((2n+4)\Gamma_i^{\frac{1}{2}}+n^2H_0^2\Big)(4\Gamma_i^{\frac{1}{2}}+n^2H_0^2),
\end{equation}
where $H_0=\sup_\Omega |{\bf H}|$. In the recent past, on the same conditions, Wang and Xia \cite{WangXia} established the following inequality
\begin{align}\label{WangXia-estimate}
    \sum_{i=1}^k(\Gamma_{k+1}-\Gamma_i)^2 \leq& \frac{1}{n}\Bigg\{\sum_{i=1}^k(\Gamma_{k+1}-\Gamma_i)^2\Big[(4+2n)\Gamma_i^{\frac{1}{2}} + n^2H_0^2\Big]\Bigg\}^{\frac{1}{2}}\nonumber\\
    &\times\Bigg\{\sum_{i=1}^k(\Gamma_{k+1}-\Gamma_i)\Big(4\Gamma_i^{\frac{1}{2}}+ n^2H_0^2\Big)\Bigg\}^{\frac{1}{2}},
\end{align}
and, applying the algebraic lemma obtained by Jost et al.~\cite{Jost} (see Lemma~\ref{jost}), they verified that the Inequality~\eqref{WangXia-estimate} implies Inequality~\eqref{Cheng-etal-estimate}, cf. \cite[Remark~2.1]{WangXia}.

Also, it is important to study estimates for lower order eigenvalues for Problem~\eqref{problem_1}. In recent times, some estimates in this sense have been obtained, for Problems~\eqref{problem2},~\eqref{problem_3} and ~\eqref{problem_1}. For more details, we can consult Cheng et al.~\cite{Cheng-etal1}, Du et al.~\cite{Du-etal} and Araújo Filho~\cite{AF} respectively, and their references.

\begin{theorem}\label{theorem1.3}
Let $\Omega$ be a bounded domain in an $n$-dimensional complete Riemannian manifold $M^n$ isometrically immersed in $\mathbb{R}^m$. Denote by $\Gamma_i$ the $i$-th eigenvalue  of Problem~\eqref{problem_1}, then we have
\begin{align}\label{theorem1-estimate1}
    \sum_{i=1}^k(\Gamma_{k+1}-\Gamma_i)^2
    \leq& \frac{1}{n\varepsilon}\Bigg\{\sum_{i=1}^k(\Gamma_{k+1}-\Gamma_i)^2\Big[(2n+4)\delta \Gamma_i^{\frac{1}{2}} + n^2H_0^2+T_0^2+4C_0\Big]\Bigg\}^{\frac{1}{2}}\nonumber\\
    &\times\Bigg\{\sum_{i=1}^k(\Gamma_{k+1}-\Gamma_i)\Big[4\delta \Gamma_i^{\frac{1}{2}} + n^2H_0^2+T_0^2+4C_0 \Big]\Bigg\}^{\frac{1}{2}},
\end{align}
and
\begin{align*}
    \sum_{i=1}^k(\Gamma_{k+1}-\Gamma_1)^\frac{1}{2}\leq& \frac{1}{\varepsilon}\Bigg[4\delta \Gamma_1^{\frac{1}{2}} + n^2H_0^2+T_0^2+4C_0 \Bigg]^{\frac{1}{2}}\nonumber\\
    &\times\Bigg[(2n+4)\delta \Gamma_1^{\frac{1}{2}} + n^2H_0^2+T_0^2+4C_0\Bigg]^{\frac{1}{2}},
\end{align*}
where
\begin{equation*}
    C_0=\sup_\Omega \left\{\frac{1}{2}\dv \Big( T\big(T(\nabla \eta)-\tr(\nabla T)\big)\Big) - \frac{1}{4}|T(\nabla \eta)|^2\right\},
\end{equation*}
$T_0=\sup_{\Omega}|\tr(\nabla T)|$, $H_0=\sup_\Omega |{\bf H}_T|$ and ${\bf H}_T$ is the generalized mean curvature vector of the immersion.
\end{theorem} 
\begin{remark} Theorem~\ref{theorem1.3} is an improvement of a result by Araújo Filho~\cite[Theorem~1.1]{AF}. Indeed, since through it one can determine a quadratic inequality in $\Gamma_{k+1}$  (see Corollary~\ref{cor_1.3}) from which one can find an estimate for $\Gamma_{k+1}$ and the gap between consecutive eigenvalues(see Corollary~\ref{cor_1.4}). Moreover, Theorem~\ref{theorem1.3} generalizes Inequality~(1.14) in Wang and Xia~\cite{WangXia}. In fact, if $T=I$ and $\eta$ is constant we have $\varepsilon=\delta=1, T_0 = C_0 = 0$ and ${\bf H}_T={\bf H}$ which is the mean curvature vector of the immersion. Consequently the first inequality of Theorem~\ref{theorem1.3} becomes the mentioned inequality obtained by Wang and Xia for Problem~\ref{problem2}.
\end{remark}

From Inequality~\eqref{theorem1-estimate1} and Lemma~\ref{jost}, we get the next result.

\begin{cor}\label{cor_1.3} 
Under the same setup as in Theorem~\ref{theorem1.3}, we have
\begin{align}\label{Equation-1.6}
    \sum_{i=1}^k(\Gamma_{k+1}-\Gamma_i)^2\leq&\frac{1}{n^2\varepsilon^2}\sum_{i=1}^k(\Gamma_{k+1}-\Gamma_i)\Big[(2n+4)\delta\Gamma_i^{\frac{1}{2}} + n^2H_0^2+T_0^2+ 4C_0 \Big]\nonumber \\
        &\times\Big(4\delta \Gamma_i^{\frac{1}{2}}+ n^2H_0^2+T_0^2+4C_0 \Big).
\end{align}
\end{cor}

Finally, we obtain an upper bound of $\Gamma_{k+1}$ solving the quadratic inequality~\eqref{Equation-1.6} for $\Gamma_{k+1}$.

\begin{cor}\label{cor_1.4} 
Under the same setup as in Theorem~\ref{theorem1.3}, we have 
\begin{align}\label{Equation-1.7}
    \Gamma_{k+1}\leq A_k + \sqrt{A_k^2-B_k},
\end{align}
in particular,
\begin{align}\label{Equation-1.8}
     \Gamma_{k+1} - \Gamma_k \leq 2\sqrt{A_k^2-B_k},
\end{align}
where 
\begin{align*}
    A_k = &\frac{1}{2n^2\varepsilon^2k}\sum_{i=1}^k\Big[(2n+4)\delta\Gamma_i^{\frac{1}{2}} + n^2H_0^2+T_0^2+4C_0 \Big]\Big(4\delta \Gamma_i^{\frac{1}{2}}+ n^2H_0^2+T_0^2+4C_0 \Big)\\
    &+\frac{1}{k}\sum_{i=1}^k\Gamma_i,
\end{align*}
and
\begin{align*}
    B_k =& \frac{1}{n^2\varepsilon^2 k}\sum_{i=1}^k\Gamma_i\Big[(2n+4)\delta\Gamma_i^{\frac{1}{2}} + n^2H_0^2+T_0^2+4C_0 \Big]\Big(4\delta \Gamma_i^{\frac{1}{2}}+ n^2H_0^2+T_0^2+4C_0 \Big)\\
   &+ \frac{1}{k}\sum_{i=1}^k\Gamma_i^2.
\end{align*}
\end{cor}
\begin{remark}
It is important to highlight that Inequality\eqref{Equation-1.6} generalizes \cite[Inequality~(1.14)]{Cheng-etal} and Corollary~\ref{cor_1.4} generalizes \cite[Corolaries~1 and 2]{Cheng-etal} both obtained by Cheng et al. in \cite{Cheng-etal}. Just take $T = I$ and $\eta = constant$ in our result. 
\end{remark}

\section{Preliminaries}

We begin by recalling some fundamental facts and the basic notation from the theory that will be used.

Here $\Omega\subset M$ is a bounded domain with smooth boundary $\partial\Omega$, $\eta$ is a real-valued function $C^2(M)$ and  $T$ is a symmetric positive definite $(1,1)$-tensor on $M$.
Let us consider $dm = e^{-\eta}d\Omega$ and $d\mu = e^{-\eta}d\partial\Omega$ the weight volume form on $\Omega$ and the volume form on the boundary $\partial\Omega$ induced by the outward unit normal vector $\nu$ on $\partial \Omega$, respectively.

We will identify a $(0,2)$-tensor $T:\mathfrak{X}(\Omega)\times\mathfrak{X}(\Omega)\to C^{\infty}(\Omega)$ with its associated $(1,1)$-tensor $T:\mathfrak{X}(\Omega)\to\mathfrak{X}(\Omega)$ by the equality
\begin{equation*}
    \langle T(X), Y \rangle = T(X, Y),
\end{equation*}
where $\mathfrak{X}(\Omega)$ is the space of smooth vector fields on $\Omega$. In particular, the tensor $\langle , \rangle$ will be identified with the identity $I$ in $\mathfrak{X}(\Omega)$.
Note that $\varepsilon I\leq T\leq \delta I$ on $\Omega$ implies
\begin{align}\label{T-property}
|T(X)|^2 \leq \delta  \langle T(X), X\rangle \quad \mbox{for all} \quad X \in \mathfrak{X}(\Omega),
\end{align}
in particular
\begin{align*}\label{T-norm}
    |T(\nabla\eta)|^2 \leq \delta^2|\nabla\eta|^2.
\end{align*}

Now, let $\alpha$ be the second fundamental form on an $n$-dimensional complete Riemannian manifold $(M^n, \langle , \rangle)$  isometrically immersed in $\mathbb{R}^m$. We know that ${\bf H}=\frac{1}{n}\tr (\alpha)$ is the mean curvature vector, then the generalized mean curvature vector, associated with a symmetric $(1, 1)$-tensor $T$, is the normal vector field defined by
\begin{equation}\label{mean-curvature}
    {\bf H}_T=\frac{1}{n}\sum_{i,j=1}^nT(e_i, e_j)\alpha(e_i, e_j)=\frac{1}{n}\sum_{i=1}^n\alpha(T(e_i), e_i):=\frac{1}{n}\tr{(\alpha\circ T)},
\end{equation}
where $\{e_1, e_2, \ldots, e_n\}$ is a local orthonormal frame of $TM$. Observe that, when $T=I$ then ${\bf H}_T={\bf H}$.

We will present some properties below, focusing on the real-valued functions and those of vector-valued functions.
    
From the definition of $\eta$-divergence (see Eq.~\eqref{EQ2.1}) and the usual properties of divergence of vector fields, one has
\begin{equation}\label{property1}
    \dv_\eta(fX)=f\dv_\eta X + \langle \nabla f, X \rangle
\end{equation}
and
\begin{equation*}
    \mathscr{L}(f\ell) = f\mathscr{L}\ell + 2 T(\nabla f, \nabla\ell) + \ell\mathscr{L}f,
\end{equation*}
for all real-valued functions $f, \ell \in C^\infty (\Omega)$.

Notice that the $(\eta,T)$-divergence form of $\mathscr{L}$ on $\Omega$ allows us to obtain the divergence theorem
\begin{equation}\label{property2}
   \int_\Omega\dv_\eta X dm = \int_{\partial\Omega} \langle X, \nu\rangle d\mu, 
   \end{equation}
in particular, for $X=T(\nabla f)$, 
\begin{equation*}
   \int_\Omega\mathscr{L}{f}dm = \int_{\partial\Omega}T(\nabla f, \nu) d\mu, 
\end{equation*}
and, the integration by parts formula
\begin{equation}\label{parts}
     \int_{\Omega}\ell\mathscr{L}{f}dm =-\int_\Omega T(\nabla\ell, \nabla f)dm + \int_{\partial\Omega}\ell T(\nabla f, \nu) d\mu,
\end{equation}
for all $\ell, f \in C^\infty(\Omega)$.

In addition, for the eigenfunction $u_i$ corresponding to the eigenvalue $\Gamma_i$, from Problem~\eqref{problem_1} and integration by parts, we have
\begin{equation}\label{lambda_i}
    \Gamma_i\int_\Omega u_i^2dm = \int_\Omega u_i\mathscr{L}^2u_idm =  \int_\Omega (\mathscr{L}u_i)^2dm.
\end{equation}

Considering ${\bf u}=(u^1, u^2, \ldots, u^n)$ a vector-valued function from $\Omega$ to $\mathbb{R}^n$, we define
\begin{eqnarray*}
\nabla {\bf u} = ( \nabla u^1, \ldots ,  \nabla u^n) \quad \mbox{and} \quad T(\nabla{\bf u}) = (T(\nabla u^1), \ldots , T(\nabla u^n)),
\end{eqnarray*}
so that
\begin{equation}\label{norm-Tu}
\|T (\nabla {\bf u})\|^2 = \int_\Omega \sum_{j=1}^n |T(\nabla u^j)|^2 dm = \int_\Omega |T(\nabla {\bf u})|^2 dm.
\end{equation}
For simplicity, we use the same notation $\|\cdot\|$ for the norm in Eq.~\eqref{norm-Tu} and for the canonical norm of a real-valued function in $L^2(\Omega,dm)$, that is, $\|{\bf u}\|^2= \displaystyle\int_\Omega |{\bf u}|^2 dm$. Besides, we are using the classical norm $|{\bf u}|^2 = \sum_{j=1}^n(u^j)^2$.

Moreover, we also see that
\begin{align}\label{definition}
    T(X, \nabla {\bf u}) &= ( T(X) \cdot \nabla { u}^1, \ldots, T(X) \cdot \nabla { u}^n )\nonumber\\
    &=( X\cdot T(\nabla u^1), \ldots, X\cdot T(\nabla u^n) )\nonumber\\
     &=(T(X, \nabla u^1), \ldots, T(X, \nabla u^n)).
\end{align}

So, the following equality is well understood for a vector-valued function ${\bf u}$ and a real-valued function $f \in C^\infty (\Omega)$
\begin{align}\label{equation2.2}
    &\mathscr{L}(f {\bf u}) = (\mathscr{L}(f u^1), \ldots, \mathscr{L}(f u^n))\nonumber\\
    &= (f\mathscr{L}u^1+2 T(\nabla f)\cdot \nabla u^1+ \mathscr{L}(f) u^1, \ldots, f\mathscr{L}u^n+2 T(\nabla f)\cdot \nabla u^n+ \mathscr{L}(f) u^n) \nonumber\\
    &=f(\mathscr{L}u^1, \ldots, \mathscr{L}u^n) + 2(T(\nabla f)\cdot \nabla u^1, \ldots, T(\nabla f)\cdot \nabla u^ne) + \mathscr{L}(f)( u^1, \ldots, u^n)\nonumber\\
    &=f\mathscr{L}{\bf u} + 2T(\nabla f, \nabla {\bf u}) + \mathscr{L}f{\bf u}.
\end{align}

 From Eqs.~\eqref{property1} and \eqref{property2}, for all vector-valued function ${\bf u}$ and ${\bf v}$ both from $\Omega$ to $\mathbb{R}^n$, with ${\bf v}$ vanishing on $\partial \Omega$,  we obtain
\begin{equation}\label{divformula}
    \int_\Omega {\bf v} \cdot \nabla(\dv_\eta {\bf u})dm = -\int_\Omega \dv_\eta {\bf u} \dv_\eta {\bf v} dm. 
\end{equation}

\section{Proof of the results about the coupled system of second-order elliptic differential equations}

In order to prove our first theorem, we will need a technical result. For this, let us first recall the following lemma. 

\begin{lem}\cite[Lemma~4.2]{AF-G}\label{lema2}
Let $\Omega\subset\mathbb{R}^n$ be a bounded domain, $\sigma_i$ be the $i$-th eigenvalue of Problem~\eqref{problem1} and ${\bf u}_i$ be a normalized vector-valued eigenfunction corresponding to $\sigma_i$. Then, for any positive integer $k$, we get
\begin{align*}
     \sum_{i=1}^k(\sigma_{k+1}-&\sigma_i)^2 \leq \frac{4(n\delta+\alpha)}{n^2\varepsilon^2} \sum_{i=1}^k(\sigma_{k+1}-\sigma_i)\Big\{\frac{1}{4}\int_{\Omega}|{\bf u}_i|^2|\tr(\nabla T)-T(\nabla \eta)|^2dm \nonumber\\
    &+ \int_\Omega {\bf u}_i\cdot{\Big (} T(\tr(\nabla T), \nabla {\bf u}_i) - T(T(\nabla \eta), \nabla {\bf u}_i){\Big )}dm + \|T(\nabla {\bf u}_i)\|^2\Big\}.
\end{align*}
\end{lem}

Now, let us rewrite the right side of inequality of the previous lemma in a more convenient way for us.
\begin{lem}\label{lemma2.3}
Under the same setup as in Lemma~\ref{lema2}, we get
\begin{align*}
     \sum_{i=1}^k(\sigma_{k+1}-\sigma_i)^2 \leq& \frac{4(n\delta+\alpha)}{n^2\varepsilon^2} \sum_{i=1}^k(\sigma_{k+1}-\sigma_i)\Big\{\|T(\nabla {\bf u}_i)\|^2 + \frac{1}{4}\int_{\Omega}|{\bf u}_i|^2| \tr(\nabla T)|^2dm \\  
     &+\int_{\Omega}|{\bf u}_i|^2\Big[\frac{1}{2}\dv \Big(T^2(\nabla \eta)-T(\tr(\nabla T))\Big) - \frac{1}{4}|T(\nabla \eta)|^2\Big] dm
   \Big\}.
\end{align*}
.
\end{lem}
\begin{proof} We make use of Lemma~\ref{lema2}. For this, we must notice that
\begin{align*}
    |\tr(\nabla T)-T(\nabla \eta)|^2 = |\tr(\nabla T)|^2 - 2\langle \tr(\nabla T), T(\nabla \eta) \rangle + |T(\nabla \eta)|^2,
\end{align*}
hence
\begin{align}\label{eqq2.23}
     \frac{1}{4}&\int_{\Omega}|{\bf u}_i|^2|\tr(\nabla T)-T(\nabla \eta)|^2 dm + \int_\Omega{\bf u}_i \cdot (T(\tr(\nabla T), \nabla {\bf u}_i) -  T(T(\nabla \eta), \nabla {\bf u}_i)))dm \nonumber \\
     =&\frac{1}{4}\int_{\Omega}|{\bf u}_i|^2|T(\nabla \eta)|^2 dm -\int_\Omega{\bf u}_i \cdot T( T(\nabla \eta), \nabla {\bf u}_i) dm  +\frac{1}{4}\int_{\Omega}|{\bf u}_i|^2|\tr(\nabla T)|^2dm\nonumber\\
     &- \frac{1}{2}\int_\Omega|{\bf u}_i|^2\langle \tr(\nabla T), T (\nabla \eta)\rangle dm +\int_\Omega{\bf u}_i \cdot T(\tr(\nabla T), \nabla {\bf u}_i)dm.
\end{align}
Since ${\bf u}_i|_{\partial \Omega}=0$, by Eq.~\eqref{definition} and the divergence theorem, we have
\begin{align}\label{Eq_3.2}
    -&\int_\Omega{\bf u}_i \cdot T(T(\nabla \eta),\nabla {\bf u}_i) dm\nonumber\\
    &= -\int_\Omega u_i^1 \langle T^2(\nabla \eta), \nabla u_i^1\rangle dm - \cdots  -\int_\Omega u_i^n \langle T^2(\nabla \eta), \nabla u_i^n\rangle dm \nonumber\\
   &= -\frac{1}{2}\int_\Omega \langle T^2(\nabla \eta), \nabla (u_i^1)^2\rangle dm - \cdots  -\frac{1}{2}\int_\Omega \langle T^2(\nabla \eta), \nabla (u_i^n)^2\rangle dm \nonumber\\
    &=\frac{1}{2}\int_\Omega (u_i^1)^2 \dv_\eta (T^2(\nabla \eta))dm + \cdots + \frac{1}{2}\int_\Omega (u_i^n)^2 \dv_\eta (T^2(\nabla \eta))dm \nonumber\\
    &=\frac{1}{2}\int_\Omega |{\bf u}_i|^2 \dv_\eta (T^2(\nabla \eta))dm,
\end{align}
ad analogously, we also have
\begin{align*}
    \int_\Omega{\bf u}_i \cdot &T(\tr(\nabla T), \nabla {\bf u}_i) dm = -\frac{1}{2}\int_\Omega |{\bf u}_i|^2 \dv_\eta (T(\tr (\nabla T)))dm\\
    &= -\frac{1}{2}\int_\Omega |{\bf u}_i|^2 \dv(T(\tr (\nabla T)))dm +  \frac{1}{2}\int_\Omega|{\bf u}_i|^2\langle \tr(\nabla T), T (\nabla \eta)\rangle dm.
\end{align*}
Substituting Eq.~\eqref{Eq_3.2} and the previous equality in Eq.~\eqref{eqq2.23}, we get

\begin{align*}
     \frac{1}{4}&\int_{\Omega}|{\bf u}_i|^2|\tr(\nabla T)-T(\nabla \eta)|^2 dm + \int_\Omega{\bf u}_i\cdot (T(\tr(\nabla T), \nabla {\bf u}_i) - T( T(\nabla \eta), \nabla {\bf u}_i))dm \nonumber \\
     =&\int_{\Omega}|{\bf u}_i|^2{\Big(}\frac{1}{4}|T(\nabla \eta)|^2 +\frac{1}{2}\dv_\eta (T^2(\nabla \eta)){\Big)}dm   +\frac{1}{4}\int_{\Omega}|{\bf u}_i|^2|\tr(\nabla T)|^2 dm \nonumber\\
     &-\frac{1}{2}\int_\Omega |{\bf u}_i|^2 \dv(T(\tr (\nabla T)))dm .
\end{align*}

Since $ \dv_\eta (T^2(\nabla \eta)) = \dv (T^2(\nabla \eta)) - |T(\nabla \eta)|^2$, by the previous equality, we have
\begin{align}\label{eqq2.24}
     &\frac{1}{4}\int_{\Omega}|{\bf u}_i|^2|\tr(\nabla T)-T(\nabla \eta)|^2 dm + \int_\Omega{\bf u}_i\cdot (T(\tr(\nabla T), \nabla {\bf u}_i) - T( T(\nabla \eta), \nabla {\bf u}_i)))dm \nonumber \\
     &=\frac{1}{4}\int_{\Omega}|{\bf u}_i|^2|\tr(\nabla T)|^2 dm + \int_{\Omega}|{\bf u}_i|^2\Big[\frac{1}{2}\dv \Big(T^2(\nabla \eta)-T(\tr(\nabla T))\Big) - \frac{1}{4}|T(\nabla \eta)|^2\Big] dm.
\end{align}
Replacing Eq.~\eqref{eqq2.24} into Lemma~\ref{lema2}, we complete the proof of Lemma~\ref{lemma2.3}.
\end{proof}

Now, we are in a position to give the proofs of the results of the Section~\ref{sec_1}.
\begin{proof}[\bf Proof of Theorem~\ref{theorem1.1}]
The proof is a consequence of Lemma~\ref{lemma2.3}. Since $T_0=\sup_{\Omega}|\tr(\nabla T)|$, $\eta_0=\sup_{\Omega}|\nabla \eta|$ and 
\begin{equation*}
    C_0=\sup_{\Omega}\Big\{\frac{1}{2}\dv \Big(T^2(\nabla \eta)-T(\tr(\nabla T))\Big) - \frac{1}{4}|T(\nabla \eta)|^2\Big\},
\end{equation*} 
we have
\begin{align}\label{EQ_4.1}
    \frac{1}{4} \int_{\Omega} |{\bf u}_i|^2 | \tr(\nabla T)|^2dm \leq \frac{1}{4}T_0^2\int_\Omega |{\bf u}_i|^2 dm =\frac{1}{4} T_0^2,
\end{align}

\begin{equation*}
    \int_{\Omega}|{\bf u}_i|^2\Big[\frac{1}{2}\dv \Big(T^2(\nabla \eta)-T(\tr(\nabla T))\Big) - \frac{1}{4}|T(\nabla \eta)|^2\Big] dm \leq C_0.
\end{equation*}
From Problem~\eqref{problem1} and Eqs.~\eqref{parts} and \eqref{divformula} we obtain 
\begin{equation*}
    \sigma_i=\int_{\Omega}T(\nabla {\bf u}_i) \cdot \nabla {\bf u}_i dm + \alpha \|\dv_\eta{\bf u}_i\|^2.
\end{equation*}
Since there exist positive real numbers $\varepsilon$ and $\delta$ such that $\varepsilon I \leq T \leq \delta I$, from the previous equality and inequality~\eqref{T-property}, we get
\begin{equation}\label{equation5-4}
    \|T(\nabla {\bf u}_i)\|^2 \leq \delta  \int_{\Omega}T(\nabla {\bf u}_i) \cdot \nabla {\bf u}_i dm = \delta(\sigma_i - \alpha \|\dv_\eta{\bf u}_i\|^2).
\end{equation}
Therefore, substituting \eqref{EQ_4.1}-\eqref{equation5-4} in Lemma~\ref{lemma2.3} we obtain
\begin{equation*}
     \sum_{i=1}^k(\sigma_{k+1}-\sigma_i)^2 \leq \frac{4(n\delta+\alpha)}{n^2\varepsilon^2} \sum_{i=1}^k(\sigma_{k+1}-\sigma_i)\Big\{\delta(\sigma_i - \alpha \|\dv_\eta{\bf u}_i\|^2)  + \frac{T_0^2}{4}+C_0\Big\},
\end{equation*}
which is enough to complete the proof.
\end{proof}

\begin{proof}[\bf Proof of Theorem~\ref{theorem1.2}]
Let $\{x_\beta \}_{\beta=1}^n$ be the standard coordinate functions of $\mathbb{R}^n$. Let us consider the matrix $D=(d_{ij})_{n \times n}$ where
\begin{equation*}
    d_{ij}:= \int_\Omega x_i{\bf u}_1\cdot {\bf u}_{j+1} dm.
\end{equation*}
Using the orthogonalization of Gram-Schmidt, we know that there exists an upper triangle matrix $R=(r_{ij})_{n \times n}$ and an orthogonal matrix $S=(s_{ij})_{n \times n}$ such that $R=SD$, namely
\begin{equation*}
    r_{ij}=\sum_{k=1}^ns_{ik}d_{kj}= \sum_{k=1}^n s_{ik} \int_\Omega x_k{\bf u}_1\cdot {\bf u}_{j+1} dm = \int_\Omega \Big( \sum_{k=1}^ns_{ik}x_k\Big){\bf u}_1\cdot {\bf u}_{j+1} dm = 0,
\end{equation*}
for $1 \leq j < i \leq n$. Putting $y_i=\sum_{k=1}^ns_{ik}x_k$, we have 
\begin{equation*}
    \int_\Omega y_i{\bf u}_1\cdot {\bf u}_{j+1} dm = 0,   \quad \mbox{for} \quad 1 \leq j < i \leq n.
\end{equation*}
Let us denote by $\displaystyle a_i=\int_\Omega y_i |{\bf u}_1|^2dm$ to consider the vector-valued functions ${\bf w}_i$ given by
\begin{equation}\label{wi}
    {\bf w}_i = (y_i - a_i){\bf u}_1,
\end{equation}
so that 
\begin{equation*}
    {\bf w}_i|_{\partial \Omega}=0 \quad \mbox{and} \quad  \int_\Omega {\bf w}_i\cdot {\bf u}_{j+1} dm =0, \quad \mbox{for any} \quad j= 1, \ldots, i-1.
\end{equation*}
Then, using the Rayleigh quotient and formula~\eqref{divformula} we obtain
\begin{equation}\label{equationn4.19}
    \sigma_{i+1} \|{\bf w}_i\|^2 \leq \int_\Omega ( -{\bf w}_i \cdot \mathscr{L}{\bf w}_i +  \alpha(\dv_\eta{\bf w}_i)^2)dm.
\end{equation}
 By \eqref{equation2.2}  we get
\begin{align}\label{equationn4.20}
    &-\int_\Omega {\bf w}_i\cdot \mathscr{L}{\bf w}_idm =  -\int_\Omega {\bf w}_i\cdot\big[(y_i - a_i) \mathscr{L}{\bf u}_1 + {\bf u}_1 \mathscr{L}y_i + 2T(\nabla y_i, \nabla {\bf u}_1)\big]dm \nonumber\\
    &=  -\int_\Omega {\bf w}_i\cdot \big[(y_i - a_i)(-\sigma_1{\bf u}_1-\alpha\nabla \dv_\eta{\bf u}_1 ) + {\bf u}_1 \mathscr{L}y_i + 2T(\nabla y_i, \nabla {\bf u}_1)\big]dm \nonumber \\
    &= \sigma_1\|{\bf w}_i\|^2 +\alpha\int_\Omega (y_i-a_i){\bf w}_i \cdot \nabla \dv_\eta{\bf u}_1 dm - \int_\Omega {\bf w}_i\cdot ({\bf u}_1 \mathscr{L}y_i + 2T(\nabla y_i, \nabla {\bf u}_1))dm.
\end{align}
Using \eqref{property1} and \eqref{property2}, with direct computation, we obtain
\begin{align*}
   \alpha \int_\Omega (y_i-a_i){\bf w}_i \cdot \nabla \dv_\eta{\bf u}_1 dm =& - \alpha \int_\Omega (\dv_\eta{\bf w}_i)^2dm \\&- \alpha \int_\Omega (\nabla (\nabla y_i\cdot {\bf u}_1) + \dv_\eta {\bf u}_1\nabla y_i)\cdot {\bf w}_i dm.
\end{align*}
Substituting the previous equality into \eqref{equationn4.20}, we get
\begin{align}\label{equationn4.21}
    \int_\Omega(-{\bf w}_i\cdot \mathscr{L}{\bf w}_i+&\alpha(\dv_\eta{\bf w}_i)^2)dm\nonumber\\
    =&\sigma_1\|{\bf w}_i\|^2 - \alpha \int_\Omega (\nabla (\nabla y_i\cdot {\bf u}_1) + \dv_\eta {\bf u}_1\nabla y_i)\cdot {\bf w}_idm \nonumber\\
    &- \int_\Omega {\bf w}_i\cdot ({\bf u}_1 \mathscr{L}y_i + 2T(\nabla y_i, \nabla {\bf u}_1))dm.
\end{align}
Replacing \eqref{equationn4.21} into \eqref{equationn4.19}, we have
\begin{align}\label{equationn4.22}
    (\sigma_{i+1} - \sigma_1)\|{\bf w}_i\|^2 \leq & 
    -\int_\Omega {\bf w}_i\cdot ({\bf u}_1 \mathscr{L}y_i + 2T(\nabla y_i, \nabla {\bf u}_1))dm \nonumber\\
    &-\alpha \int_\Omega {\bf w}_i \cdot (\nabla (\nabla y_i\cdot {\bf u}_1) + \dv_\eta {\bf u}_1\nabla y_i) dm.
\end{align}
By a straightforward computation, we have, from \eqref{property1}, \eqref{property2}, \eqref{parts} and \eqref{wi},
\begin{align}\label{equationn4.23}
    - \int_\Omega {\bf w}_i\cdot ({\bf u}_1 \mathscr{L}y_i + 2T(\nabla y_i, \nabla {\bf u}_1))dm =  \int_\Omega |{\bf u}_1|^2T(\nabla y_i, \nabla y_i) dm,
\end{align}
and
\begin{align}\label{equationn4.24}
    - \alpha \int_\Omega {\bf w}_i \cdot (\nabla (\nabla y_i\cdot {\bf u}_1) + \dv_\eta {\bf u}_1\nabla y_i)dm = \alpha\int_\Omega |\nabla y_i \cdot {\bf u}_1|^2dm.
\end{align}
Therefore, substituting \eqref{equationn4.23} and \eqref{equationn4.24} into \eqref{equationn4.22} we obtain
\begin{align}\label{equationn4.25}
    (\sigma_{i+1} - \sigma_1)\|{\bf w}_i\|^2 \leq \int_\Omega |{\bf u}_1|^2T(\nabla y_i, \nabla y_i) dm + \alpha\int_\Omega |\nabla y_i \cdot {\bf u}_1|^2dm.
\end{align}
From \eqref{equationn4.23}, for any constant $B >0$,  we infer
\begin{align*}
    (\sigma_{i+1} - \sigma_1)&\int_\Omega |{\bf u}_1|^2T(\nabla y_i, \nabla y_i) dm \\
    =&  (\sigma_{i+1} - \sigma_1)\Big\{ - 2\int_\Omega {\bf w}_i\cdot \Big(\frac{1}{2}{\bf u}_1 \mathscr{L}y_i + T(\nabla y_i, \nabla {\bf u}_1)\Big)dm\Big\} \\
    \leq& 2 (\sigma_{i+1} - \sigma_1)\|{\bf w}_i\|\Big\| \frac{1}{2}{\bf u}_1 \mathscr{L}y_i + T(\nabla y_i, \nabla {\bf u}_1) \Big\| \\
    \leq& B(\sigma_{i+1} - \sigma_1)^2\|{\bf w}_i\|^2 + \frac{1}{B}\Big\| \frac{1}{2}{\bf u}_1 \mathscr{L}y_i + T(\nabla y_i, \nabla {\bf u}_1) \Big\|^2,
\end{align*}
hence using \eqref{equationn4.25} and the previous inequality we get
\begin{align}\label{equationn4.26}
    (\sigma_{i+1} - \sigma_1)&\int_\Omega |{\bf u}_1|^2T(\nabla y_i, \nabla y_i) dm \nonumber\\
\leq& B(\sigma_{i+1} - \sigma_1) \Big( \int_\Omega |{\bf u}_1|^2T(\nabla y_i, \nabla y_i) dm + \alpha\int_\Omega |\nabla y_i \cdot {\bf u}_1|^2dm \Big) \nonumber\\
    & +\frac{1}{B}\Big\| \frac{1}{2}{\bf u}_1 \mathscr{L}y_i + T(\nabla y_i, \nabla {\bf u}_1) \Big\|^2.
\end{align}
Summing over $i$ from $1$ to $n$ in \eqref{equationn4.26}, we conclude that
\begin{align}\label{equationn4.27}
\sum_{i=1}^n&(\sigma_{i+1} - \sigma_1)(1-B)\int_\Omega |{\bf u}_1|^2T(\nabla y_i, \nabla y_i) dm \nonumber\\ &\leq B\alpha\sum_{i=1}^n(\sigma_{i+1} - \sigma_1)\int_\Omega |\nabla y_i \cdot {\bf u}_1|^2dm
+\frac{1}{B}\sum_{i=1}^n\Big\| \frac{1}{2}{\bf u}_1 \mathscr{L}y_i + T(\nabla y_i, \nabla {\bf u}_1) \Big\|^2.
\end{align}
From the definition of $y_i$ and the fact that $S$ is an orthogonal matrix, we know that $\{y_i\}_{i=1}^n$ are also the coordinate functions in $\mathbb{R}^n$. Then, as in the proof of~\cite[Theorem~1.1]{AF-G} (see inequality 4.18)  we can also get
\begin{align*}
   0 < \sum_{i=1}^n\Big\| \frac{1}{2}{\bf u}_1 \mathscr{L}y_i &+ T(\nabla y_i, \nabla {\bf u}_1) \Big\|^2 \leq \Bigg\{\frac{1}{4}\int_{\Omega}|{\bf u}_1|^2|\tr(\nabla T)-T(\nabla \eta)|^2 dm\nonumber\\
   &+ \!\int_\Omega\!{\bf u}_1\cdot{\Big (}T(\tr(\nabla T), \!\nabla {\bf u}_1) - T(T(\nabla \eta), \!\nabla {\bf u}_1){\Big )}dm + \|T (\nabla {\bf u}_1)\|^2 \!\Bigg\}.
\end{align*}
Analogous to the proofs of Lemma~\ref{lemma2.3} and Theorem~\ref{theorem1.1}, we have

\begin{align}\label{equationn4.15}
    \Bigg\{\frac{1}{4}\int_{\Omega}|{\bf u}_1|^2|\tr(\nabla T)-T(\nabla \eta)|^2 dm& + \|T (\nabla {\bf u}_1)\|^2\nonumber\\
   + \int_\Omega{\bf u}_1\cdot{\Big (}T(\tr(\nabla T), \nabla {\bf u}_1) &- T(T(\nabla \eta), \nabla {\bf u}_1){\Big )}dm  \Bigg \} \nonumber\\
   &\leq \Big\{\delta(\sigma_1 - \alpha \|\dv_\eta{\bf u}_1\|^2)  + \frac{T_0}{4}+C_0\Big\}.
\end{align}

Substituting \ref{equationn4.15} into \ref{equationn4.27} and using $\varepsilon I \leq T \leq \delta I$, we obtain

\begin{align}\label{equationn4.28}
    \sum_{i=1}^n(\sigma_{i+1} - \sigma_1)(\varepsilon - B(\delta + \alpha)) \leq \frac{1}{B}\Big\{\delta(\sigma_1 - \alpha \|\dv_\eta{\bf u}_1\|^2)  + \frac{T_0}{4}+C_0\Big\}.
\end{align}
Since $B$ is an arbitrary positive constant, we can take 
\begin{equation*}
    B=\Bigg\{\frac{\delta(\sigma_1 - \alpha \|\dv_\eta{\bf u}_1\|^2)  + \frac{T_0}{4}+C_0}{(\delta+ \alpha)\sum_{i=1}^n(\sigma_{i+1} - \sigma_1)}\Bigg\}^{\frac{1}{2}}
\end{equation*}
into \eqref{equationn4.28} and therefore we get 
\begin{equation}\label{equation5.16}
    \sum_{i=1}^n(\sigma_{i+1} - \sigma_1) \leq \frac{4(\delta+\alpha)}{\varepsilon^2}\Big\{  \delta(\sigma_1 - \alpha \|\dv_\eta{\bf u}_1\|^2)  + \frac{T_0}{4}+C_0\Big\}.
\end{equation}
We can take $i = 1$ in inequality \eqref{equation5-4} and replace in \eqref{equation5.16} to obtain Theorem~\ref{theorem1.2}.
\end{proof}

\begin{proof}[\bf Proof of Corollary~\ref{Cor_1}]
By setting $v_i := \sigma_i + D_0$ into Theorem~\ref{Cor_1} we get
\begin{equation}\label{equation1.13}
    \sum_{i=1}^k(v_{k+1}-v_i)^2 \leq \frac{4\delta(n\delta+\alpha)}{ n^2\varepsilon^2}\sum_{i=1}^k(v_{k+1}-v_i)v_i. 
\end{equation}
We know that each $v_i>0$, and from~\eqref{equationn1.3}, $v_1 \leq v_2 \leq \cdots \leq v_{k+1}$. Notice that \eqref{equation1.13} is a quadratic inequality of $v_{k+1}$. Then, we can follow the steps of the proof of Inequalities~(1.3) and (1.4) in Gomes and Miranda \cite{GomesMiranda} to obtain the inequalities of the corollary.
\end{proof}

\section{Proof of the results about the fourth-order elliptic operators}

In this section, the keystone to our proofs relies on a slight modification from \cite[Lemma 3.2]{AF} by Araújo Filho, namely:

\begin{lem}\label{lemma_2}
Let $\Omega$ be a bounded domain in an $n$-dimensional complete Riemannian manifold $M$ isometrically immersed in $\mathbb{R}^m$, $\Gamma_i$ be the $i$-th eigenvalue of Problem~\eqref{problem_1} and $u_i$ be a normalized real-valued eigenfunction corresponding to $\Gamma_i$. Then we have
\begin{align}\label{Equation-3.3}
   \sum_{i=1}^k (\Gamma_{k+1}-\Gamma_i)^2 \int_\Omega u_i^2\tr (T)dm\leq& 4B\sum_{i=1}^k(\Gamma_{k+1}-\Gamma_i)^2\Big(q_i - \frac{1}{2}\int_\Omega u_i \mathscr{L}u_i \tr{(T)}dm\Big)\nonumber\\
   &+\frac{1}{B} \sum_{i=1}^k (\Gamma_{k+1}-\Gamma_i)q_i,
\end{align}
for any positive constant $B$, where
\begin{align*}
  q_i=&\|T (\nabla u_i)\|^2_{L^2} +\frac{n^2}{4}\int_{\Omega}u_i^2|{\bf H}_T|^2dm +\frac{1}{4}\int_\Omega u_i^2 |\tr(\nabla T)|^2 dm\nonumber\\
   & +\int_\Omega u_i^2\Big\{\frac{1}{2}\dv \Big[T\Big(T(\nabla \eta) - \tr(\nabla T)\Big)\Big] - \frac{1}{4}|T(\nabla \eta )|^2\Big\}dm,
\end{align*}
and ${\bf H}_T$ is the generalized mean curvature vector of the immersion.
\end{lem}

\begin{proof} 
From Lemma 3.2 in \cite{AF} we have
\begin{align*}
   \sum_{i=1}^k (\Gamma_{k+1}-\Gamma_i)^2 \int_\Omega u_i^2\tr(T)dm\leq& 4B\sum_{i=1}^k(\Gamma_{k+1}-\Gamma_i)^2\Big(q_i - \frac{1}{2}\int_\Omega u_i \mathscr{L}u_i \tr{(T)}dm\Big)\nonumber\\
   &+\frac{1}{B} \sum_{i=1}^k (\Gamma_{k+1}-\Gamma_i)q_i,
\end{align*}
where
\begin{align}\label{Equation-qi_2}
  q_i=&\|T (\nabla u_i)\|^2_{L^2} +\frac{n^2}{4}\int_{\Omega}u_i^2|{\bf H}_T|^2dm+\int_\Omega u_i^2\Big(\frac{1}{4}|T(\nabla \eta)|^2+\frac{1}{2}\dv_\eta (T^2(\nabla \eta)) \Big)dm\nonumber\\
   & + \int_\Omega u_i T(\tr(\nabla T), \nabla u_i)dm +\frac{1}{4}\int_\Omega u_i^2\langle\tr(\nabla T), \tr(\nabla T) - 2T(\nabla \eta)\rangle dm.
\end{align}

Since $u_i|_{\partial \Omega}=\frac{\partial u_i}{\partial \nu_{T}}|_{\partial \Omega}=0$ from divergence theorem, we have
\begin{align}\label{Q1}
    &\int_\Omega u_i T(\tr(\nabla T),\nabla u_i) dm = \int_\Omega u_i \langle T(\tr(\nabla T)), \nabla u_i\rangle dm \nonumber\\
   &= \frac{1}{2}\int_\Omega \langle T(\tr(\nabla T)), \nabla u_i^2\rangle dm =-\frac{1}{2}\int_\Omega u_i^2 \dv_\eta (T(\tr(\nabla T)))dm\nonumber\\
   &=-\frac{1}{2}\int_\Omega u_i^2 \dv (T(\tr(\nabla T)))dm + \frac{1}{2}\int_\Omega u_i^2 \langle \tr (\nabla T), T(\nabla \eta )\rangle dm,
\end{align}
and
\begin{align}\label{Q2}
    \int_\Omega u_i^2\Big(\frac{1}{4}|T(\nabla \eta)|^2+&\frac{1}{2}\dv_\eta (T^2(\nabla \eta)) \Big)dm \nonumber\\
    &= \int_\Omega u_i^2\Big(\frac{1}{4}|T(\nabla \eta)|^2+\frac{1}{2}\dv (T^2(\nabla \eta)) - \frac{1}{2}|T(\nabla \eta )|^2\Big)dm\nonumber\\
    &=\int_\Omega u_i^2\Big(\frac{1}{2}\dv (T^2(\nabla \eta)) - \frac{1}{4}|T(\nabla \eta )|^2\Big)dm.
\end{align}

Substituting \eqref{Q1} and \eqref{Q2} into \eqref{Equation-qi_2} we have

\begin{align*}
  q_i=&\|T (\nabla u_i)\|^2_{L^2} +\frac{n^2}{4}\int_{\Omega}u_i^2|{\bf H}_T|^2dm+\int_\Omega u_i^2\Big(\frac{1}{2}\dv (T^2(\nabla \eta)) - \frac{1}{4}|T(\nabla \eta )|^2\Big)dm\nonumber\\
   & -\frac{1}{2}\int_\Omega u_i^2 \dv (T(\tr(\nabla T))dm + \frac{1}{2}\int_\Omega u_i^2 \langle \tr (\nabla T), T(\nabla \eta )\rangle dm\nonumber\\
   &+\frac{1}{4}\int_\Omega u_i^2 |\tr(\nabla T)|^2 dm - \frac{1}{2}\int_\Omega u_i^2\langle\tr(\nabla T), T(\nabla \eta)\rangle dm,
\end{align*}
that is, 
\begin{align*}
  q_i=&\|T (\nabla u_i)\|^2_{L^2} +\frac{n^2}{4}\int_{\Omega}u_i^2|{\bf H}_T|^2dm\nonumber\\
   &+\int_\Omega u_i^2\Big\{\frac{1}{2}\dv \Big[T\Big(T(\nabla \eta) - \tr(\nabla T)\Big)\Big] - \frac{1}{4}|T(\nabla \eta )|^2\Big\}dm +\frac{1}{4}\int_\Omega u_i^2 |\tr(\nabla T)|^2 dm,
\end{align*}
which is sufficient to complete the proof of Lemma~\ref{lemma_2}. 
\end{proof}

\begin{proof}[\bf Proof of Theorem\ref{theorem1.3}]
The proof of the first inequality is a consequence of Lemma~\ref{lemma_2}. Let us denote 
\begin{equation*}
    C_0=\sup_\Omega \left\{\frac{1}{2}\dv \Big( T\big(T(\nabla \eta)-\tr(\nabla T)\big)\Big) - \frac{1}{4}|T(\nabla \eta)|^2\right\},
\end{equation*}
$T_0=\sup_{\Omega}|\tr(\nabla T)|$ and $H_0=\sup_\Omega |{\bf H}_T|$ so that 
\begin{equation}\label{4.5}
    \frac{n^2}{4}\int_\Omega u_i^2|{\bf H}_T|^2dm \leq \frac{n^2H_0^2}{4}\int_\Omega u_i^2dm = \frac{n^2H_0^2}{4},
\end{equation}
\begin{align}
    \int_\Omega u_i^2\Big\{\frac{1}{2}\dv \Big[T\Big(T(\nabla \eta) - \tr(\nabla T)\Big)\Big] - \frac{1}{4}|T(\nabla \eta )|^2\Big\}dm \leq C_0
\end{align}
and
\begin{align}\label{t0}
    \frac{1}{4}\int_\Omega u_i^2 |\tr(\nabla T)|^2 dm \leq \frac{T_0^2}{4}.
\end{align}
From \eqref{4.5}-\eqref{t0}, we get
\begin{align}\label{Equation-4.7}
    q_i &\leq \|T(\nabla u_i)\|_{L^2}^2 + \frac{n^2H_0^2+4C_0+ T_0^2}{4}.
\end{align}

 Since there exist positive real numbers $\varepsilon$ and $\delta$ such that $\varepsilon I \leq T \leq \delta I$, consequently $n\varepsilon \leq \tr{(T)} \leq n\delta$, hence
\begin{align}\label{eq_6.6}
    n\varepsilon = n\varepsilon \int_\Omega u_i^2dm \leq \int_\Omega \tr{(T)}u_i^2 dm,
\end{align}
and using \eqref{lambda_i}, we have
\begin{align}\label{Equation-4.6}
    -\frac{1}{2}\int_\Omega u_i \mathscr{L}u_i \tr{(T)} dm &\leq \frac{1}{2} \Bigg(\int_\Omega u_i^2 dm \Bigg)^{\frac{1}{2}}\Bigg(\int_\Omega (\tr{(T))^2(\mathscr{L}u_i)^2}dm\Bigg)^{\frac{1}{2}}\nonumber\\
    &\leq \frac{n\delta}{2}\Bigg(\int_\Omega (\mathscr{L}u_i)^2dm\Bigg)^{\frac{1}{2}} = \frac{n\delta}{2}\Gamma_i^{\frac{1}{2}}.
\end{align}
Moreover, using \eqref{lambda_i} once again, notice that
\begin{align*}
    \|T(\nabla u_i)\|_{L^2}^2 = \int_\Omega \langle T(\nabla u_i), T (\nabla u_i) \rangle dm \leq \delta \int_\Omega T(\nabla u_i, \nabla u_i) = - \delta \int_\Omega u_i \mathscr{L}u_i dm,
\end{align*}
therefore, we have
\begin{align}\label{Equation(4.8)}
    \|T(\nabla u_i)\|_{L^2}^2\leq - \delta \int_\Omega u_i \mathscr{L}u_i dm \leq \delta \Bigg( \int_\Omega u_i^2 dm\Bigg)^{\frac{1}{2}}\Bigg( \int_\Omega (\mathscr{L}u_i)^2 dm\Bigg)^{\frac{1}{2}} = \delta \Gamma_i^{\frac{1}{2}}.
\end{align}
Thus, from \eqref{Equation-4.7} and \eqref{Equation(4.8)} we obtain
\begin{align}\label{Equation-4.8}
    q_i &\leq \delta \Gamma_i^{\frac{1}{2}}  + \frac{n^2H_0^2+4C_0+ T_0^2}{4}.
\end{align}
Substituting \eqref{eq_6.6}, \eqref{Equation-4.6} and \eqref{Equation-4.8} into \eqref{Equation-3.3}, we have
\begin{align*}
     n\varepsilon\sum_{i=1}^k(\Gamma_{k+1}-\Gamma_i)^2 \leq& 4B\sum_{i=1}^k(\Gamma_{k+1}-\Gamma_i)^2\Big( \frac{n+2}{2}\delta\Gamma_i^{\frac{1}{2}} +  \frac{n^2H_0^2+4C_0+T_0^2}{4}\Big)\nonumber\\
     & + \frac{1}{B}\sum_{i=1}^k(\Gamma_{k+1}-\Gamma_i)\Big(\delta \Gamma_i^{\frac{1}{2}} + \frac{n^2H_0^2+4C_0+T_0^2}{4}\Big).
\end{align*}
Finally, taking
\begin{align*}
    B=\frac{\Big\{\sum_{i=1}^k(\Gamma_{k+1}-\Gamma_i)\Big(\delta \Gamma_i^{\frac{1}{2}} + \frac{n^2H_0^2+4C_0+T_0^2}{4}\Big)\Big\}^{\frac{1}{2}}}{\Big\{\sum_{i=1}^k(\Gamma_{k+1}-\Gamma_i)^2\Big( \frac{n+2}{2}\delta\Gamma_i^{\frac{1}{2}} +  \frac{n^2H_0^2+4C_0+T_0^2}{4}\Big)\Big\}^{\frac{1}{2}}},
\end{align*}
into previous inequality, we complete the proof of Inequality~\eqref{theorem1-estimate1}.

The proof of the second part of the theorem is a slight modification of the proof of \cite[Inequality (1.10)]{AF}. In fact, analogously to what we did in the Lemma~\ref{lemma_2}, we just need to make the following modification in \cite[Inequality~(4.12)]{AF}
\begin{align*}
\int_\Omega\Big(u_1^2 \big(n^2 |{\bf H}_T|^2& + |\tr{(\nabla T)}|^2 - 2 \langle \tr{(\nabla T)}, T(\nabla \eta)\rangle + |T(\nabla \eta)|^2) \nonumber\\
&+ 4u_1\big(T(\tr{(\nabla T)}, \nabla u_1) - T(T(\nabla \eta), \nabla u_1)\big) + 4|T(\nabla u_1)|^2\Big)dm\nonumber\\
\leq &4\|T (\nabla u_1)\|^2_{L^2} +n^2\int_{\Omega}u_1^2|{\bf H}_T|^2dm +\int_\Omega u_1^2 |\tr(\nabla T)|^2 dm\nonumber\\
   & +4\int_\Omega u_1^2\Big\{\frac{1}{2}\dv \Big[T\Big(T(\nabla \eta) - \tr(\nabla T)\Big)\Big] - \frac{1}{4}|T(\nabla \eta )|^2\Big\}dm \nonumber\\
\leq& 4\|T(\nabla u_1)\|_{L^2}^2 + n^2H_0^2 + 4C_0 + T_0^2,
\end{align*}
To complete the proof, simply follow the rest of the proof by replacing the \cite[Inequality~(4.12)]{AF} by the previous Inequality.
\end{proof}

To prove Corollary~\ref{cor_1.3}, we needed an algebraic lemma obtained by Jost et al.
\begin{lem}{~\cite[Lemma~2.3]{Jost}}\label{jost}
Let $\{a_i\}_{i=1}^m$, $\{b_i\}_{i=1}^m$ and $\{c_i\}_{i=1}^m$ be three sequences of non-negative real with $\{a_i\}_{i=1}^m$ decreasing and $\{b_i\}_{i=1}^m$ and $\{c_i\}_{i=1}^m$ increasing. Then we have
\begin{equation*}
    \Big(\sum_{i=1}^m a_i^2b_i \Big)\Big(\sum_{i=1}^m a_ic_i \Big)\leq \Big(\sum_{i=1}^m a_i^2\Big)\Big(\sum_{i=1}^m a_ib_ic_i\Big).
\end{equation*}
\end{lem}

\begin{proof}[\bf Proof of Corollary~\ref{cor_1.3}]
Since $\{\Lambda_i=(\Gamma_{k+1} - \Gamma_i)\}_{i=1}^k$ is decreasing and $\{(2n+4)\delta\Gamma_i^{\frac{1}{2}} + n^2H_0^2+T_0^2+4C_0\}_{i=1}^k$ and $\{4\delta \Gamma_i^{\frac{1}{2}}+ n^2H_0^2+T_0^2+4C_0\}_{i=1}^k$ are increasing, we obtain from the previous lemma that
\begin{align}\label{Equation-1.5}
    \Bigg\{\sum_{i=1}^k&\Lambda_i^2\Big[(2n+4)\delta\Gamma_i^{\frac{1}{2}} + n^2H_0^2+T_0^2+4C_0 \Big]\Bigg\}\Bigg\{\sum_{i=1}^k\Lambda_i\Big(4\delta\Gamma_i^{\frac{1}{2}} + n^2H_0^2+T_0^2+4C_0 \Big)\Bigg\}\nonumber\\
    \leq&\Bigg\{\sum_{i=1}^k\Lambda_i\Big[(2n+4)\delta\Gamma_i^{\frac{1}{2}} + n^2H_0^2+T_0^2+4C_0 \Big]\Big(4\delta\Gamma_i^{\frac{1}{2}} + n^2H_0^2+T_0^2+4C_0 \Big)\Bigg\}\nonumber\\
    &\times  \Bigg(\sum_{i=1}^k\Lambda_i^2\Bigg).
\end{align}
Substituting \eqref{Equation-1.5} into \eqref{theorem1-estimate1}, we complete the proof of the corollary.     
\end{proof}

\begin{proof}[\bf Proof of Corollary~\ref{cor_1.4}]
Since our Inequality~\eqref{Equation-1.6} is a quadratic inequality of $\Gamma_{k+1}$, it is not difficult to obtain \eqref{Equation-1.7}, that is,
\begin{equation}\label{Equation-1.7-2}
    \Gamma_{k+1} \geq A_k + \sqrt{A_k^2-B_k}.
\end{equation}
Now, notice that Inequality~\eqref{Equation-1.6} also holds if we replace the integer $k$ with $k-1$, that is, we have
\begin{align*}
    &\sum_{i=1}^{k-1}(\Gamma_{k}-\Gamma_i)^2\nonumber \\
    &\leq\frac{1}{n^2\varepsilon^2}\sum_{i=1}^{k-1}(\Gamma_{k}-\Gamma_i)\Big[(4\delta^2+2n\delta)\Gamma_i^{\frac{1}{2}} + n^2H_0^2+4\delta^2C_0 \Big]\Big(4\delta^2 \Gamma_i^{\frac{1}{2}}+ n^2H_0^2+4\delta^2C_0 \Big),
\end{align*}
hence, we infer
\begin{align*}
    &\sum_{i=1}^{k}(\Gamma_{k}-\Gamma_i)^2\nonumber \\
    &\leq\frac{1}{n^2\varepsilon^2}\sum_{i=1}^{k}(\Gamma_{k}-\Gamma_i)\Big[(4\delta^2+2n\delta)\Gamma_i^{\frac{1}{2}} + n^2H_0^2+4\delta^2C_0 \Big]\Big(4\delta^2 \Gamma_i^{\frac{1}{2}}+ n^2H_0^2+4\delta^2C_0 \Big).
\end{align*}
Therefore, $\Gamma_k$ also satisfies the same quadratic inequality and we obtain
\begin{equation*}
    \Gamma_k \geq A_k - \sqrt{A_k^2-B_k}.
\end{equation*}
Thus, from the previous inequality and \eqref{Equation-1.7-2} we get \eqref{Equation-1.8} and complete the proof of Corollary~\ref{cor_1.4}.
\end{proof}

\section*{Statements and Declarations}

{\bf Data availability statement:} the manuscript has no associated data.

{\bf Funding statement:} not applicable.

{\bf Ethics and Consent to Participate declarations:} not applicable.

{\bf Conflict of interest statement:} the authors state that there is no conflict of
interest.

\section*{Acknowledgements}

The authors have been partially supported by Fundação de Amparo à Pesquisa do Estado do Amazonas (FAPEAM) and Coordenação de Aperfeiçoamento de Pessoal de Nível Superior (CAPES).

\end{document}